\documentclass[a4paper,12pt]{amsart}
\usepackage{amsthm}
\usepackage{amssymb}
\usepackage{amsmath}

\newtheorem{theorem}{Theorem}

\newtheorem{lemma}[theorem]{Lemma}
\newtheorem{prop}[theorem]{Proposition}
\newtheorem{cor}[theorem]{Corollary}

\theoremstyle{definition}

\theoremstyle{plain}

\def\al{\alpha}
\def\la{\lambda}
\def\be{\beta}
\def\De{\Delta}
\def\lcm{\operatorname{lcm}}
\newcommand{\ep}[2]{\equiv{#1}\pmod{#2}}
\newcommand{\ve}[2]{\binom{#1}{#2}}

\date{}
\title[Lucas sequence terms divisible by their indices]{The terms in Lucas sequences divisible by their indices}
\author{ Chris Smyth}
\address{School of Mathematics and Maxwell Institute for Mathematical Sciences, University of Edinburgh,
    James Clerk Maxwell Building, King's Buildings,
    Mayfield Road, Edinburgh EH9 3JZ, UK.}
\begin{document}
\maketitle
\begin{abstract}
For Lucas sequences of the first kind $(u_n)_{n\ge 0}$ and second kind $(v_n)_{n\ge 0}$ defined as usual by $u_n=(\al^n-\be^n)/(\al-\be)$, $v_n=\al^n+\be^n$, where $\al$ and $\be$ are either integers or conjugate quadratic integers, we describe the sets $\{n\in\mathbb N:n \text{ divides } u_n\}$ and $\{n\in\mathbb N:n \text{ divides } v_n\}$. Building on earlier work, particularly that of Somer, we show that the numbers in these sets can be written as a product of a so-called {\it basic} number, which can only be $1$, $6$ or $12$, and particular primes, which are described explicitly. Some properties of the set of all  primes
that arise in this way is also given, for each kind of sequence.
\end{abstract}
\section{Introduction}
Given integers $P$ and $Q$, let $\al$ and $\be$ be the roots of the equation 
$$
x^2-Px+Q=0.
$$
 Then the well-known {\it Lucas sequence of the first kind} (or {\it generalised Fibonacci sequence}) $(u_n)_{n\ge 0}$ is given by  $u_0=0,u_1=1$ and $u_{n+2}=Pu_{n+1}-Qu_n$ for $n\ge 0$, or explicitly by  Binet's formula
$$u_n=\frac{\al^n-\be^n}{\al-\be}
$$
when $\De=(\al-\be)^2=P^2-4Q\ne 0$, and $u_n=n\al^{n-1}$ when $\De=0$. In this latter case $\al$ is an integer, and so $n$ divides $u_n$ for all $n\ge 1$. In Theorem \ref{S-Basic} below we describe, for all pairs $(P,Q)$,  the set $S=S(P,Q)$ of all $n\ge 1$ for which $n$ divides $u_n$.

Corresponding to Theorem \ref{S-Basic}  we have a similar result (Theorem \ref{T-Basic}  below)  for the {\it Lucas sequence of the second kind}  $(v_n)_{n\ge 0}$, given by  $v_0=2,u_1=P$ and $v_{n+2}=Pv_{n+1}-Qv_n$ for $n\ge 0$, or explicitly by the formula
$$
v_n=\al^n+\be^n,
$$
finding the set $T=T(P,Q)$ of all $n\ge 1$ for which $n$ divides $v_n$. The results  for the set $T$ are given in Section \ref{S-T}.

For $n\in S$, define $\mathcal P_{S,n}$ to be the set of primes $p$ such that $np\in S$.  We call an element $n$ of $S$  {\it (first kind) basic} if there is no prime $p$ such that $n/p$ is in $S$.  We shall see that, for given $P,Q$,  there are at most two basic elements of $S$. It turns out that all elements of $S$ are generated from {basic} elements using  primes from these sets.

\begin{theorem} \label{S-Basic} \begin{enumerate}

\item[(a)]For $n\in S$, the set $\mathcal P_{S,n}$ is the set of primes dividing $u_n\De$. 

\item[(b)] Every element of $S$ can be written in the form $bp_1\dots p_r$ for some $r\ge 0$, where $b\in S$ is basic and, for $i=1,\dots,r$, the numbers $bp_1\dots p_{i-1}$ are also in $S$, and $p_i$ is in $\mathcal P_{S,bp_1\dots p_{i-1}}$.
\item[(c)] The (first kind) basic elements of $S$ are
\begin{itemize}
\item $1$ and $6$ if $P\ep{3}{6}$, $Q\ep{\pm 1}{6}$;
\item $1$ and $12$ if $P\ep{\pm 1}{6}$, $Q\ep{-1}{6}$;
\item $1$ only, otherwise.
\end{itemize}
\end{enumerate}
\end{theorem}

Note that the primes in part (b) need not be distinct.

Somer \cite[Theorem 4]{MR1271392} has many  results in the direction of this theorem. In particular, he already noted the importance of $6$ and $12$ for this problem.  Walsh \cite[unpublished]{Wa} gave an equivalent categorization of  $S(1,-1)$ (the Fibonacci numbers case), where $1$ and $12$ are the  basic elements of $S(1,-1)$.

Note that if $\al$ and $\be$ are integers, then at least one of $P,Q$ is even, so that $1$ is the only basic element in this case. In this case, too, it is known (see Andr{\'e}-Jeannin \cite{MR1131414}) that $S=\{n\,:\,n\mid\alpha^n-\beta^n\}$. (His result is stated assuming that $(n,\alpha\beta)=1$, and his proof given for $n$ square-free).  This follows straight from Proposition \ref{P-AJ} below.

Now let $\mathcal P_S$ be the set of primes $p$ that divide some $n$ in $S$. It is easy to see that $\mathcal P_S=\cup_{n\in S}\mathcal P_{S,n}$.
It is interesting to compare $\mathcal P_{S,n}$ and $\mathcal
P_{S,np}$ for $n$ and $np$ in $S$. Write $u_n=u_n(\al,\be)$ to show
the dependence of $u_n$  on $\al$ and $\be$, and denote
$u_n(\al^k,\be^k)$ by $u_n^{(k)}$. Then since
\begin{equation}\label{E-prod} u_{kn}=u_k^{(n)}u_n,
\end{equation}
we have $u_{n}\mid u_{np}$, so that $\mathcal P_{S,n}\subset \mathcal
P_{S,np}$ by Theorem \ref{S-Basic}(b). Thus when we multiply $n\in S$
by a succession of primes according to Theorem \ref{S-Basic}(b) to
stay within $S$, the associated set $\mathcal P_{S,n}$ does not lose any
primes. Hence we obtain the following consequence of
 Theorem \ref{S-Basic}(a).

\begin{cor} \label{C-mn} If $n\in S$ and all prime factors of $m$ divide $u_n\De$, then $nm\in S$.
\end{cor}

This is a strengthening of the known result (see e.g., \cite[Theorem 5(i)]{MR1271392}) that if $n\in S$ and $m$ all
prime factors of $m$ divide $n\De$, then $nm\in S$. In particular ($n=1$) $\Delta\in S$ and, for $n\in S$, both
$u_n=n\cdot (u_n/n)\in S$ and $u_n\Delta\in S$.

In Section \ref{S-finite} we give the conditions on $P$ and $Q$ that make $S$, $\mathcal P_S$, $T$ or $\mathcal P_T$ finite. In Section \ref{S-divisibility} we briefly discuss divisibility properties of the sequences $S$ and $T$. These properties are useful for generating the sequences efficiently.

It is of interest to estimate $\{n\in S: n\le x\}$ and $\{n\in T: n\le x\}$. It is planned to do this in a forthcoming paper of Shparlinski and the author.  For $\mathcal P_S$ infinite (and not the set $\mathcal P$ of all primes!)  it would also  of interest to estimate the relative density of $\mathcal P_S$ in $\mathcal P$. But this seems to be a more difficult problem (as does the corresponding problem for $T$).

For an interesting survey of many results on Lucas numbers, see Ribenboim \cite{MR1352481}. For a more general reference on recurrence sequences see the book \cite{MR1990179} by Everest, van der Poorten, Shparlinski, and Ward.
\section{ Preliminary results for $S$.}

While Theorem \ref{S-Basic}(b) allows us to multiply $n\in S$ by
the primes in $\mathcal P_{S,n}$ to stay within $S$, a vital ingredient in proving
Theorem \ref{S-Basic}(c) is to be able to do the opposite: to divide
$n\in S$ by a prime and stay within $S$. This is provided by the
following significant result, due to Somer, generalising special
cases due to Jarden \cite[Theorem E]{MR0197383}, Hoggatt and
Bergum \cite{MR0349567} and Walsh \cite{Wa} for the Fibonacci  sequence (i.e., $P=1$,
$Q=-1$) and Andr\'e-Jeannin \cite{MR1131414} for $\gcd(P,Q)=1$.

\begin{theorem}[{{Somer \cite[Theorem 5(iv)]{MR1271392}}}]\label{S-down} Let $n\in S$, $n>1$, with $p_{\max\nolimits}$ its
largest prime factor. Then, except in the case that $P$ is odd and $n$ is of the form $2^\ell\cdot 3$ for some
$\ell\ge 1$, we have  $n/p_{\max\nolimits}\in S$.
\end{theorem}

We produce a variant of  this result to cover all but two of the exceptional cases, as follows.

\begin{prop} If $P$ is odd and $n=2^\ell\cdot 3\in S$, where $\ell\ge 3$, then $n/2\in S$.
\end{prop}

The idea of the proof of  Theorem \ref{S-down} is roughly (i.e., ignoring some details)  as follows. Let $n$ have prime factorization $n=\prod_pp^{k_p}$, with $\omega(n)$, the {\it rank of appearance of $n$}, being the least integer $k$ such that $n\mid u_k$. Then $n\mid u_n$ is equivalent to $\omega(n)\mid n$. Since $\omega(n)=\lcm_p\omega\left(p^{k_p}\right)$, and every $\omega\left(p^{k_p}\right)$ is of the form $p^{k'_p}\ell_p$, where $k'_p<k_p$ and $\ell_p\mid(p^2-1)$, it follows that $n\mid u_n$ is equivalent to
\begin{equation}\label{E-omega}
\lcm_{p\mid n}\left(p^{k'_p}\ell_p\right)\mid n=\prod_{p\mid n}p^{k_p}.
\end{equation}
But since for $p>2$ all prime factors of $p^2-1$ are less than $p$, and $2^2-1=3$, if equation (\ref{E-omega}) holds, it will still hold with $n$ replaced by $n/p_{\max\nolimits}$ when $p_{\max\nolimits}>3$ or $p_{\max\nolimits}=3$ and ($n$ odd or $2\mid n$ with $\ell_2=1$). When $p_{\max\nolimits}=3$ and $2\mid n$ with $\ell_2=3$, (\ref{E-omega})
will still hold with $n$ replaced by $n/3$ as long as $n/3$ is divisible by $3$.

For the proof of Theorem \ref{S-Basic}, we first need the following, which  dates back to Lucas \cite[page 295]{MR1505176} and Carmichael \cite[Lemma II]{MR1502458}. It is the special case $n=1$ of Theorem \ref{S-Basic}(a).
\begin{lemma} \label{L-pD} For any prime $p$, $p$ divides $u_p$ if and only if $p$ divides $\De$.
\end{lemma}
\begin{proof} Now $u_2=P$ and $\De=P^2-4Q\ep{u_2}{2}$, so the result is true for $p=2$. The result is trivial for $\De=0$. Now for $\De\ne 0$ and $p\ge 3$,
\begin{align*}
\De^{(p-1)/2}&=\frac{(\al -\be)^{p}}{(\al-\be)}\\
&=u_p+\sum_{j=1}^{p-1}\binom{p}{j}\al^{p-j}(-\be)^j/(\al-\be)\\
&=u_p+\sum_{j=1}^{(p-1)/2}\binom{p}{j}(-1)^jQ^ju_{p-2j}\\
&\ep {u_p}{p},
\end{align*}
giving the result.
\end{proof}

We have the following.

A prime is called {\it irregular} if it divides $Q$ but not $P$. Clearly $p\nmid \Delta$ for $p$ irregular. A prime that is not irregular is called {\it regular}.
\begin{lemma}[{{Lucas \cite[pp. 295--297]{MR1505176}}},  {{Carmichael \cite[Theorem XII]{MR1502458}}}, {{Somer \cite[Proposition 1(viii)]{MR1271392}}}]
\label{L-eps}
If $p$ is an odd prime with $p\nmid Q$, $p\nmid \Delta$, then $p\mid
u_{p-\varepsilon}$, where $\varepsilon$ is the Legendre symbol
$\left(\frac{\Delta}{p}\right)$. On the other hand, if $p$ is irregular then it does not divide any $u_n$, $n\ge 1$.
\end{lemma}

The following result follows straight for Lemmas \ref{L-pD} and \ref{L-eps}.

\begin{cor} \label{C-reg} The set $\mathcal P_{{\operatorname{1st}}}$ of primes that divide some $u_n$, $n\ge 1$ consists precisely of the regular primes.
\end{cor}

\begin{lemma}[{{Somer \cite[Theorem 5(ii)]{MR1271392}}}]\label{L-lcm}
 If $m,n\in S$ then $\lcm(m,n)\in S$.
 \end{lemma}
 \begin{proof} Put $\ell=\lcm(m,n)$. From  (\ref{E-prod}) we have $u_n\mid u_\ell$, $u_m\mid u_\ell$, so $n\mid u_n$, $m\mid u_m$ and hence $\ell\mid u_\ell$.
\end{proof}
\begin{lemma}
\label{L-easy}
If $P$ and $Q$ are integers and $p$ is a prime not dividing $\gcd(P,Q)$ then there  is an integer $P^*\ep{P}{p}$ such that $\gcd(P^*,Q)=1$.
\end{lemma}
 \begin{proof} If $p\nmid P$ then choose $k$ so that $P^*=P+kp$ is a prime greater than $Q$, while if $p\mid P$  choose $k$ so that $P^*=p(P/p+k)$ is a $p$ times a prime greater than $Q$.  
\end{proof}

\begin{lemma}\label{L-12684}
We have
\begin{enumerate} 
\item[(i)]  If $P$ is odd and $2^\ell\mid u_{12}$ then $2^{\ell-1}\mid u_6$;
\item[(ii)] If $3\mid u_{8k}$ then $3\mid u_{4k}$.
\end{enumerate}
\end{lemma}

\begin{proof} Using the notation
$$
 P^{(k)}=P(\al^k,\be^k)=\al^k+\be^k=v_k,\qquad
Q^{(k)}=Q(\al^k,\be^k)=Q^k,
$$
we have $P^{(2)}=P^2-2Q$ and 
\begin{equation}\label{E-P4}
P^{(4)}=(P^2-2Q)^2-2Q^2=P^4-4P^2Q+2Q^2.
\end{equation}
\begin{enumerate} 
\item[(i)] Take $P$ odd. Then  $$P^{(2)}\equiv\begin{cases} \quad 1\pmod{4} \text{ if $Q$ even}\\
-1\pmod{4} \text{ if $Q$ odd}\end{cases}, $$
and so $P^{(4)}\ep{P^{(2)}}{4}$ and
 $$
v_6=P^{(2)}(P^{(4)}-Q^2) \equiv\begin{cases} 1\pmod{4} \text{ if $Q$ even}\\
2\pmod{4} \text{ if $Q$ odd}\end{cases}.
$$
Since $u_{12}=u_6v_6$ by (\ref{E-prod}), we get the result.
\item[(ii)] Since $u_{4k}=u_k^{(4)}u_4$, it is enough to prove that if $3\mid u_{2k}^{(4)}$ and $3\nmid u_4$ then $3\mid u_{k}^{(4)}$. Now, working modulo $3$, $P^{(4)}\equiv P^2(1-Q)-Q^2$, 
using (\ref{E-P4}) and $P^4\equiv P^2$. Thus
$$
\ve{P^{(4)}}{Q^{(4)}}=\begin{cases} \ve{0}{0} \text{ if } P\equiv Q\equiv 0\\
\ve{1}{0} \text{ if } P\equiv\pm 1, Q\equiv 0\\
\ve{1}{1} \text{ if } P\equiv\pm 1, Q\equiv -1\\
\ve{-1}{1} \text{ otherwise.}\end{cases}
$$
The result holds in the first case because $u_4\equiv 0$, and in the second case  because  $u_n^{(4)}\equiv 1$ for all $n\ge 1$. In the other two cases, $u_n^{(4)}\equiv 0$ precisely when $3\mid n$, so the result holds also in these cases.
\end{enumerate}
\end{proof}

\begin{prop} \label{P-23} If $P$ is odd and $2^\ell\cdot 3\in S$, where $\ell\ge 3$, then $2^{\ell-1}\cdot 3\in S$. In particular, then $12\in S$.
\end{prop}

\begin{proof} Take $P$ odd. Then $P^{(2)}=P^2-2Q$ is also odd, and hence so are all $P^{(2^\ell)}=v_{2^\ell}$ for $\ell\ge 0$. Then for $\ell\ge 3$, using (\ref{E-prod}) and $u_{2k}=u_kv_k$ we have
$$
u_{2^\ell\cdot 3}=u_{12}^{(2^{\ell-2})}u_{2^{\ell-2}}=u_{12}^{(2^{\ell-2})}v_{2^{\ell-3}}v_{2^{\ell-4}}\dots v_2v_1.
$$
So if $2^\ell\mid u_{2^\ell\cdot 3}$ then $2^\ell\mid u_{12}^{(2^{\ell-2})}$ so, by Lemma \ref{L-12684}(i), $2^{\ell-1}\mid u_6^{(2^{\ell-2})}$. Hence
$$
2^{\ell-1}\mid u_6^{(2^{\ell-2})}u_{2^{\ell-2}}=u_{2^{\ell-1}\cdot 3}.
$$
 Also, if $3\mid u_{2^\ell\cdot 3}$ where $\ell\ge 3$  then $3\mid u_{2^{\ell-1}\cdot 3}$, by Lemma \ref{L-12684}(ii). Thus we have proved that if $\ell\ge 3$ and $2^\ell\cdot 3\in S$ then $2^{\ell-1}\cdot 3\in S$. Then $12\in S$ follows.
\end{proof}

\begin{prop}\label{P-AJ}
For any positive integer $n$ and distinct integers $a,b$,

$$
n\mid a^n-b^n\Longrightarrow n\mid \frac{a^n-b^n}{a-b}.
$$
\end{prop}
\begin{proof} For any prime $p$, suppose that $p^\ell\|a-b$ and $p^r\|n$. It is clearly enough to prove that $p^{r+\ell}\mid a^n-b^n$ whenever $\ell>0$. Put $a=b+\la p^\ell$. Then
\begin{align*}
a^n-b^n&=\sum_{k=1}^n\binom{n}{k}\la^kp^{\ell k}b^{n-k}\\
&=\sum_{k=1}^n\frac{n}{k}\binom{n-1}{k-1}\la^kp^{\ell k}b^{n-k}\\
&\equiv 0\pmod{p^L},
\end{align*}
where
\begin{align*}
L&\ge r+\min_{k=1}^n (\ell k-\lfloor\log_p k\rfloor)\\
&\ge r+\ell +\min_{k=1}^n (\ell(k-1)-\log_2 k)\\
&\ge r+\ell+\min_{k=1}^n ((k-1)-\log_2 k)\\
&=r+\ell.
\end{align*}
\end{proof}

\section{Proof of Theorem \ref{S-Basic}.}

To prove part (a), take
$n\in S$ and $p$ prime.
First note that, from (\ref{E-prod}), $u_{np}=u_p^{(n)}u_n$. 
Now suppose that $np\mid u_{np}$. Then either $p\mid u_n$, or, by
Lemma \ref{L-pD},
 we have $p\mid \De^{(n)}$, where $\De^{(n)}=(\al^n-\be^n)^2= u_n^2\De $. Hence $p\mid u_n\De $.

Conversely, suppose $p\mid  u_n\De$. Then $p\mid \De^{(n)}$, so
that, by Lemma \ref{L-pD}, $p\mid u_p^{(n)}$, giving $pn\mid
u_p^{(n)}u_n=u_{np}$.

To prove (b), take $n\in S$, $n\ne 1,6$ or $12$. If $3\in S$ then $3/3=1\in S$. Otherwise, by Theorem \ref{S-down}
and Proposition \ref{P-23}, we have $n/p\in S$ for some prime factor $p$ of $n$.  Thus we obtain a sequence
$n, n/p, (n/p)/p',\dots$ of elements of $S$, which stops only at $1$, $6$ or $12$. But clearly $6$ and $12$ cannot both
 be basic, so the process will stop at either $1$ (always basic!) or at most one of $6$ and $12$. This shows that this sequence, written backwards, must be of the form $b, bp_1, bp_1p_2,\dots, bp_1\dots p_r$, say, as required. By (a), we know that 
 $p_i$ is in $\mathcal P_{S,bp_1\dots p_{i-1}}$. 

To prove (c), we just need to find for which $P,Q$ the numbers $6$ or $12$ are basic.

 {\bf The case \mbox{\boldmath$6\in S, 3\not\in S, 2\not\in S$}.} Since
$u_2=P$, we know that $2\in S$ iff $P$ is even. Hence $P$ is odd.
Also

\begin{align}\label{E-u633}
u_6=u_3v_3&=(P^2-Q)(P^2-3Q)P.
\end{align}
As $6\mid u_6$ and $3\nmid u_3=P^2-Q$, we have $3\mid P$, and so $Q\ep{\pm 1}{3}$. Also $Q$ must be odd, so $P\ep{3}{6}$ and $Q\ep{\pm 1}{6}$.

{\bf The case \mbox{\boldmath$12\in S, 6\not\in S, 4\not\in S$}.}  Since $2\not\in S$ by Corollary \ref{C-mn}, we have $P$ odd, as above. Now $u_{12}=u_6v_6$ and
\begin{align}\label{E-v633}
v_6=v_3^{(2)}&=(P^2-2Q)((P^2-2Q)^2-3Q^2).
\end{align}
If $Q$ were even,  then by (\ref{E-u633}) and (\ref{E-v633}) $u_6$, $v_6$, and $u_{12}$ would all be odd. So $Q$ is odd.
As $u_6$ is then even, $3\nmid u_6$, and we have $P\ep{\pm 1}{3}$ and $Q\ep{0\text{ or }  -1}{3}$. As $3\mid u_{12}$, also $3\mid v_6\ep{(P^2-2Q)^3}{3}$, giving $Q\ep{-1}{3}$. Hence $P\ep{\pm 1}{6}$ and $Q\ep{-1}{6}$.

The converse for both of these cases is easily checked.

\section{The set $T$}\label{S-T}

 The results for the set $T=\{n\in\mathbb N: n\mid v_n\}$ differ slightly from those for $S$. Essentially, this is because of difficulties at the prime $2$: $v_n$ divides $v_{np}$ for $p$ odd, but not in general for $p=2$. The main result is the following.
For $n\in T$, define $\mathcal P_{T,n}$ to be the set of primes $p$
such that $np\in T$. A prime is said to be {\it special} if it divides
both $P$ and $Q$. It is clear from applying the recurrence relation that
all $v_n$ for $n\ge 1$ are divisible by  $\gcd(P,Q)$, and so by all
special primes. We say that an element $n$ of $T$  is {\it (second
kind) basic} if there is no prime $p$ such that $n/p$ is in $T$.

\begin{theorem} \label{T-Basic} \begin{enumerate}

\item[(a)] For $n\in T$, the set $\mathcal P_{T,n}$ is the set of odd primes dividing $v_n$, with the possible inclusion of $2$. Specifically, the prime $2$ is in $\mathcal P_{T,n}$ if and only if $n$ is a product of special primes and either
\begin{itemize}
\item $P$ is even;

or
\item $Q$ is odd and $3\mid n$.
\end{itemize}
\item[(b)] Every element of $T$ can be written in the form $bp_1\dots p_r$
for some $r\ge 0$, where $b\in T$ is (second kind) basic  and, for
$i=1,\dots,r$, the numbers $bp_1\dots p_{i-1}$ are also in $T$, and
$p_i$ is in $\mathcal P_{bp_1\dots p_{i-1}}$.
\item[(c)] The (second kind) basic elements of $T$ are
\begin{itemize}
\item $1$ and $6$ if $P\ep{\pm 1}{6}$, $Q\ep{-1}{6}$;
\item $1$ only, otherwise.
\end{itemize}
\end{enumerate}
\end{theorem}

As in Theorem \ref{S-Basic}, the primes in part (b) of Theorem
\ref{T-Basic} need not be distinct. Note that part (a) of the theorem implies that, unless $2$ is special, no element of $T$ is divisible by $4$. Again, Somer  \cite[Theorem 4]{MR1393479}
had many results concerning the set $T$. In particular, he already
noted  the importance of $6$  for its
structure.

We now compare $\mathcal P_{T,n}$ and $\mathcal P_{T,np}$, as we did
 $\mathcal P_{S,n}$ and $\mathcal P_{S,np}$. But, in this case,
the prime $2$ is, unsurprisingly, anomalous.

\begin{cor}\label{C-Q}
\begin{enumerate} \item[(a)] For an odd prime $p$ in $\mathcal P_{T,n}$, we
have $p\in\mathcal P_{T,np}$;
\item[(b)] For $q$ an odd prime with $q\in\mathcal P_{T,n}$, we have $q\in\mathcal P_{T,2n}$ if and only if $q\mid Q$;
\item[(c)] For $2\in\mathcal P_{T,n}$, we have $2\in\mathcal P_{T,2n}$ if and only if $2$ is special.
\end{enumerate}
\end{cor}
\begin{proof} Part (a) follows from the fact that for $p$ odd $v_n\mid v_{np}$ , combined with Theorem \ref{T-Basic}(a). For (b), we know from Theorem \ref{T-Basic}(a) that $q\mid v_n$. Then from $v_{2n}=v_n^2-2Q^n$ we see that $q\mid v_{2n}$ iff $q\mid Q$. For (c), we know from Theorem \ref{T-Basic}(a) that  for $2\in\mathcal P_{T,2n}$ all prime divisors of $2n$ are special, so $2$ is special.
 Conversely, if $2$ is special, then all prime factors of $2n$ are special, and $P$ is even, so that, by Theorem \ref{T-Basic}(a), $2\in\mathcal P_{T,2n}$.
\end{proof}
\begin{cor}\label{C-nm}
If $n\in T$ and
\begin{itemize} \item all odd prime factors of $m$ divide $v_n$;

and
\item if $m$ is even then every prime divisor of $2n$ is special;
\end{itemize}
then $nm\in T$.
\end{cor}
\begin{proof}
On successively multiplying $n$ by first the odd and then the even prime divisors of $m$, we see from Theorem \ref{T-Basic}(a) that the stated conditions ensure that we stay within $T$ while doing this.
\end{proof}

This result extends Theorem 5(i)  of Somer \cite{MR1393479}, which has the condition that `$m$ is a product of special primes or divides $n$' instead of `all odd prime factors of $m$ divide $v_n$'.

\section{Preliminary results for $T$.}

We first quote the important result of Somer for $T$, corresponding to his result (Theorem \ref{S-down} above) for $S$.

\begin{theorem}[{{Somer \cite[Theorem 5]{MR1393479}}}]\label{T-down}  Theorem \ref{S-down} holds with the set $S$ replaced by the set $T$.
\end{theorem}

Jarden \cite[Theorem E]{MR0197383} proved this result for the classical Lucas sequence (i.e., $P=1$, $Q=-1$) under the restriction $p_\text{max}\ne 3$.

\begin{lemma} \label{L-2q} 
Suppose $q$ is a special prime. Then $q^{e_n}\mid v_n$, where
$e_n\ge \lfloor\log_q n\rfloor$.
\end{lemma}
\begin{proof}
From the recurrence, it is easy to see that we can take 
$$
e_n=\begin{cases} \left\lfloor\frac{n}{2}\right\rfloor+1 \text{ if } q=2\\
\left\lfloor\frac{n+1}{2}\right\rfloor \text{ if } q\ge 3,
\end{cases}$$
the slightly higher value for $q=2$ coming from the fact that $v_0=2$. Then use
$\left\lfloor\log_q n\right\rfloor\le\left\lfloor\frac{n+1}{2}\right\rfloor$.
\end{proof}

We then immediately obtain the following.

\begin{cor}[{Special case of {Somer \cite[Theorem 5(i)]{MR1393479}}}]\label{C-spec} If $n$ is a product of special primes then it belongs to $T$.
\end{cor}

We can now extend  Theorem \ref{T-down} as follows.
\begin{prop} \label{P-32} If $\ell\ge 2$ and $2^\ell\cdot 3\in T$, then $2^\ell\in T$.
\end{prop}
\begin{proof} Put $L=2^\ell$. If $2$ is special, then,  by Corollary \ref{C-spec}, $L\in T$ for all $\ell\ge 1$. So we can assume that $2$ is not special. We then know that $Q$ must odd, as if it were even then we would have $2\mid v_{3L}\ep{P^{3L}}{Q}$, so $P$ would be even and $2$ special.

From $L\mid v_{3L}=v_L(v_L^2-3Q^L)$ we see that if $v_L$ were odd then, as $L$ is even, $Q^L$ is a square, and so $v_L^2-3Q^L\ep{2}{4}$, giving $2^1\|v_{3L}$, a contradiction. Hence $v_L$ is even, and $L\mid v_L$.
\end{proof}

Next, we consider  the set $\mathcal P_T$
of primes that divide some $n\in T$. To set our result in context, we first need the following standard result concerning the prime divisors of the set of all Lucas numbers of the second kind. This essentially dates back to Lucas  (\cite{MR1505161}, \cite{MR1505164}, \cite{MR1505176}).   See Somer \cite[Proposition 2(iv)]{MR1393479}.

\begin{prop}\label{P-p|v_n} The set of odd prime numbers that divide some $v_n$ consists of the odd special primes, as well as all those odd nonspecial primes that do not divide $Q$ and do not divide $u_k$ for any odd $k$. Furthermore $2$ divides some $v_n$ with $n\ge 1$ if and only if $Q$ is odd.
\end{prop}

\begin{proof} First note that all special primes divide all Lucas numbers $u_n$ for $n>1$. Next, if $p$ divides $Q$ but not $P$, then $v_n\ep{P^{n}}{p}$. So suppose $p$ is a nonspecial prime that does not divide $u_k$ for any $k$ odd. Now, since it is known (see \cite[p. 51]{MR1352481}) that a  prime $p$ that does not divide $Q$ divides some $u_n$, we must have $n$ even, say $n=2^rk$, with $k$ odd. Then
$$
p\mid u_n=u_kv_kv_{2k}v_{2^2k}\dots v_{2^{r-1}k}
$$
and since $p$ does not divide $u_k$, it must divide some $v_{{2^j}k}$.

Conversely, suppose that the odd prime $p$ divides some $v_n$. In the case $\gcd(P,Q)=1$, we have by \cite[equation (2.13)]{MR1352481} that $\gcd(u_k,v_n)=1$ or $2$ for $k$ odd. Hence $p$ cannot divide any $u_k$ with $k$ odd. In the general case $\gcd(P,Q)>1$ we apply the same result to the Lucas sequences $(u_n^*)$, $(v_n^*)$ with parameters $P^*$ and $Q$, where $P^*$ is as in Lemma \ref{L-easy}. Since these new sequences are congruent to the old ones mod $p$, we have for $k$ odd that $u_k\equiv u_k^*\not\equiv 0\pmod{p}$.

The result for the prime $2$ comes from  \cite[p. 50]{MR1352481}.
\end{proof}

Denote by $\mathcal P_{\operatorname{2nd}}$ the primes dividing some $v_n$, as described by the previous proposition.

Clearly $\mathcal
P_T$ is a subset of $\mathcal P_{\operatorname{2nd}}$. As for $P_S$ in $\mathcal P_{\operatorname{1st}}$, it would be interesting to prove that it is always a proper subset. Indeed, it again seems pretty clear why this should be the case. Take $p\in\mathcal P_{\operatorname{2nd}}$, not dividing $Q$,  with $p$ having even rank of appearance $\omega(p)$ (in $(u_n)$). Then $p\mid v_n$ precisely when $n$ is an odd multiple of $\omega(p)/2$ -- see Somer \cite[Proposition 2(vii)]{MR1393479}. Thus if $\omega(p)$ has an odd prime divisor $q$ that is not in $\mathcal P_{\operatorname{2nd}}$, and $q\mid n$, then we cannot possibly have $n\mid v_n$. So it remains only to prove that there always are such primes. It seems clear, for instance by looking at examples (like those in  Section \ref{S-Ex}), that there will always be many of these primes, resulting in $\mathcal P_T$ being a thin subset of $\mathcal P_{\operatorname{2nd}}$. But a proof of this is lacking at present.

Our next lemma is an easy exercise. Dickson \cite[pp.67, 271]{MR0245499} traces the result back to an `anonymous writer' in 1830 \cite{Anon}, and also to Lucas \cite[p. 229]{MR1505164}.

\begin{lemma} \label{L-binom} For $p$ an odd prime and $j=1,2,\dots,(p-1)/2$, the expression $B_j:=\binom{p-1}{j}-(-1)^j$ is divisible by $p$. 
\end{lemma}

The following result dates back to Lucas \cite[p. 210]{MR1505164} and Carmichael \cite[Theorem X]{MR1516755}.
\begin{lemma} \label{L-pP} 
\begin{enumerate}
\item[(i)] For $n\in\mathbb N$ and any prime $p$, $p$ divides $v_{np}$ if and only if $p$ divides $v_n$.
\item[(ii)] For $n\in\mathbb N$ and any odd prime $p$, $v_n$ divides $v_{np}$ and $v_{np}/v_n\ep{v_n^{p-1}}{p}$. 
\end{enumerate}
\end{lemma}
\begin{proof} \begin{enumerate}
\item[(i)] Now $v_2=v_1^2-2Q$, which is even iff $v_1$ is even.   Also, for $p\ge 3$,
\begin{align}\label{E-pP}
v_1^p=(\al +\be)^{p}
&=v_p+\sum_{j=1}^{(p-1)/2}\binom{p}{j}Q^jv_{p-2j}\ep{v_p}{p}.
\end{align}
Now replace $\al,\be$ by $\al^n,\be^n$.
\item[(ii)] Taking $p$ odd and $B_j$ defined as in Lemma \ref{L-binom}, we have
\begin{align*}
v_p&=(\al+\be)(\al^{p-1}-\al^{p-2}\be+\dots+\be^{p-1})\\
&=(\al+\be)((\al+\be)^{p-1}-\sum_{j=1}^{p-2}B_j\al^{p-1-j}\be^j)\\
&=v_1\left(v_1^{p-1}-\sum_{j=1}^{(p-3)/2}B_jQ^jv_{p-1-2j}-B_{(p-1)/2}Q^{(p-1)/2}\right).
\end{align*}
so that the result of $p$ odd follows by replacing $\al,\be$ by $\al^n,\be^n$ and using Lemma \ref{L-binom}.
\end{enumerate}
\end{proof}

\section{Proof of Theorem \ref{T-Basic}}
We now prove part (a) of Theorem \ref{T-Basic}. First take $p$ odd and $n\in T$. Then, by Lemma \ref{L-pP}(i), if $p\nmid v_n$ then $p\nmid v_{np}$, so $np\not\in T$. Conversely, if $p^\lambda\| v_n$ for some $\lambda\ge 1$ then by Lemma \ref{L-pP}(ii) $p^{\lambda+1}\mid v_{np}$. Since $n\mid v_n$ and $v_n\mid v_{np}$ we have $np\in T$.

Now take $p=2$, and suppose that both $n$ and $2n$ are in $T$. First note that $v_n$ must be even, as otherwise $v_{2n}=v_n^2-2Q^n$ would be odd. Also, we  have $n\mid Q^n$, so that every prime factor $q$ of $n$ divides $Q$. (Note that this works too if $q=2$, as then $4\mid v_{2n}$.) But $q$ must also divide $P$, as otherwise $v_n\equiv P^n\not\equiv 0\pmod{q}$. Hence $q$ is special, and $n$ is a product of special primes. If $n$ is even, then $2$ is special, so $P$ and $Q$ are both even. Alternatively, because $v_n$ is even, we must have either $P$ even and $Q$ odd or (from the recurrence) $P$ and $Q$ both odd and $3\mid n$. So we have either $P$ even or $Q$ odd and $3\mid n$.

 Conversely, assume that $n\in T$ is a product of special primes, and either $P$ is even or ($Q$ is odd and $3\mid n$). We know from Corollary \ref{C-spec} that every product of special primes is in $T$. So if $2$ is special, then $2n\in T$. So we can assume $2$ is not special, and hence that $n$ is odd. If $P$ is even, then, from the recurrence, all the $v_k$, in particular $v_n$ and $v_{2n}$, are even. Also, if $P$ and $Q$ are both odd
 and  $3\mid n$, then $v_n$ and $v_{2n}=v_n^2-2Q^n$ are both even. Since for every prime factor $q$ of $n$  with $q^\lambda\|n$ we have $\lambda\le \log_q n<n$, so that $n\mid Q^n$. Hence $2n\mid v_{2n}$, $2n\in T$.

To prove part (b): we see easily from Theorem \ref{T-down} and Proposition \ref{P-32}
 that the only possible (second
kind) basic numbers are $1$ and $6$. To find the conditions on $P$
and $Q$ that make $6$ basic, we assume that $6\in T$ but $2\notin
T$, $3\notin T$. Then $v_2=P^2-2Q$ is odd, so $P$ odd. Also $3\nmid
v_3=P(P^2-3Q)$, so $P\ep{\pm 1}{6}$. From $6\mid
v_6=v_2(v_2^2-3Q^2)$ we have $Q$ odd and $3\mid v_2\ep{1-2Q}{3}$, so
that $Q\ep{-1}{6}$. Conversely, if $P\ep{\pm 1}{6}$ and $Q\ep{-1}{6}$ then it is easily checked that $6$ is basic. This proves part (b).

The proof of part (c) is just the same as that for Theorem \ref{S-Basic}(c).

\section{Finiteness results for $S$ and for $T$.}\label{S-finite}
In this section we look at when $S$, $\mathcal P_S$, and $T$, $\mathcal P_T$ are finite. The results given here are essentially reformulations of results of Somer \cite{MR1271392}, \cite{MR1393479}.

\begin{theorem} \label{infinite} The set $S$ is finite if and only if $\De=1$, in which case $S=\{1\}$. For $S$ infinite, $\mathcal P_S$ is finite when $Q=0$ and $P\ne 0$, in which case $\mathcal P_S$ consists of the prime divisors of $P$. Otherwise, $\mathcal P_S$ is also infinite. 
Furthermore,
$\mathcal P_S$  is the set $\mathcal P$ of all primes if and only if
every prime divisor of $Q$ is special. (This includes the case $Q=\pm 1$.)
\end{theorem}

For the proof, we note first that when $\De=1$, $\al$ and $\be$ are consecutive integers, and $1$ is the only basic element.
But there are no primes $p$ dividing $u_1\De=1$, so $\mathcal P_1$ is empty, and $S=\{1\}$. In all other
cases, $|u_1\De|>1$,  $\mathcal P_{S,1}$ is nonempty, with $p\in\mathcal P_{S,1}$ say, and then, by Corollary \ref{C-mn}, $p^k\in S$ for all $k\ge 0$,   making $S$ infinite.

Now assume $S$ is infinite. We recall that $(u_n)_{n\ge 0}$ is called {\it degenerate} if $Q=0$
or $\al/\be$ is a root of unity. (The latter alternative includes the case $P=0$, $Q\ne 0$.) We consider the two cases $(u_n)$ degenerate or nondegenerate separately.  If  $(u_n)$ is degenerate, then by \cite[Theorem 9]{MR1271392}
either
\begin{itemize}
\item  $P\ne 0$ and $Q=0$, so that then $S$ consists of those $n$ whose prime factors all divide $P$, and $\mathcal P_S$ is the set of  prime divisors of $P$;

or

\item for some $r=1,2,3,4$ or $6$, $S$ has a subset $rk\quad(k\in \mathbb N)$ where $u_{rk}=0$, so that $\mathcal P_S=\mathcal P$.
\end{itemize}

Now consider the case of $(u_n)$ nondegenerate. Then, by Somer \cite[Theorem
1]{MR1271392}, all but finitely many $u_n$ have a primitive prime divisor (a prime dividing
$u_n$ that do not divide $u_m$ for any $m<n$). So, using Theorem \ref{S-Basic}(a), $\mathcal P_S$ is infinite. Somer's theorem is based on results of Lekkerkerker \cite{MR0055373} and Schinzel \cite{MR0139567}. In fact Bilu, Hanrot
and Voutier \cite{MR1863855} have proved that for such sequences with no special primes every $u_n$ with
$n>30$ has a primitive divisor. They also listed exceptions with
$n\le 30$. Hence $u_{p^k}$ has a primitive prime divisor for all
sufficiently large $k$, making $\mathcal P_S$ infinite.   See  Abouzaid \cite{MR2289425} for corrections to their
list. Also Stewart \cite{MR0491445} and Shorey and Stewart \cite{MR602235} gave lower bounds for the largest prime divisor of $u_n$. 
We mention in passing a  contrasting  result of Everest,
Stevens, Tamsett and Ward \cite{MR2309982}, who exhibited a cubic
linear recurrence for which infinitely many of the resulting
sequence had no primitive divisor.

This proof will be complete after we have proved the following. 
While this result is contained in Somer \cite[Theorem
8]{MR1271392}, we give another proof here.
\begin{prop}\label{P-allprimes} The set $\mathcal P_S$ is the whole of $\mathcal P$ if and
only if all primes are regular.
\end{prop}

\begin{proof}
First note that if there are any irregular primes then, by Corollary \ref{C-reg}, $\mathcal P_S$, being a subset of $\mathcal P_{\operatorname{1st}}$, cannot be the whole of $\mathcal P$.

Conversely, assume all primes are regular, so that any prime factor $p$ of $Q$ also divide $P$.
Note that then $p\mid \Delta$.
To show that all primes belong to $\mathcal P_S$, we proceed by
induction. We first show that $2\in\mathcal P_S$. If $u_2=P$ is
even, then $2\in S$, $2\in\mathcal P_S$. So we can take $P$ odd. Then
$Q$ must be odd, too, by our assumption. Then $u_3=P^2-Q$ is even,
and hence so is $u_6=u_3v_3$.We claim that either $3\mid u_6$,
in which case $6\in S$, $2,3\in \mathcal P_S$, or $12\in S$, with
the same implication.
\begin{itemize}
\item If $P\ep{3}{6}$, $Q\ep{3}{6}$, then $3\mid u_n$ for all $n\ge
2$, so that $3\mid u_6$.
\item If $P\ep{3}{6}$, $Q\ep{\pm 1}{6}$, then $6$ is basic, by
Theorem \ref{S-Basic}(c).
\item If $P\ep{\pm 1}{6}$, $Q\ep{-1}{6}$, then $12$ is basic, by
Theorem \ref{S-Basic}(c).
\item If $P\ep{\pm 1}{6}$, $Q\ep{1}{6}$, then $3\mid u_3$ and so
$3\mid u_3v_3=u_6$.
\end{itemize}
Hence $2\in\mathcal P_S$, as claimed.

We now assume that $q\in \mathcal P_S$ for every prime $q<p$, where
$p$ is a prime at least $3$. We have just shown that this is true
for $p=3$. By Lemma \ref{L-lcm}, we know that there is a positive integer $k$ such that $k\prod_{q<p}q\in S$; hence, by Corollary \ref{C-mn}, $k\prod_{q<p}q^{e_q}\in S$ for any exponents $e_q$.

By Lemma \ref{L-eps}, $p\mid u_{p+\varepsilon}$, where $\varepsilon=\pm 1$. As $p>2$, all factors of $p+\varepsilon$ are less than $p$ so, by the induction hypothesis, $k(p+\varepsilon)\in S$ for some $k$. Now put $k'=k/p$ if $p\mid k$, and $k'=k$ otherwise. Then, using (\ref{E-prod}), we have
$$
u_{pk'(p+\varepsilon)}=u_p^{(k'(p+\varepsilon))}u_{k'(p+\varepsilon)}=u_{pk'}^{(p+\varepsilon)}u_{p+\varepsilon},
$$
so that $pk'(p+\varepsilon)\in S$, $p\in\mathcal P_S$. This proves the
induction step.
\end{proof}

On the other hand, if there are irregular primes, then in general $\mathcal P_S$ will be a proper subset of $\mathcal P_{\operatorname{1st}}$. For an idea of why this should be the case, take an irregular prime $f$, and suppose that $p$ is a prime whose rank of appearance $\omega(p)$ is a multiple of $f$. Then if $n=kp$ were in $S$,
we would have $u_{kp}\ep{0}{p}$, so that $\omega(p)$, and hence $f$, divides $kp$. Hence $f$ divides $u_n$, a contradiction. Thus we have shown that no prime whose rank of appearance is a multiple of $f$ will belong to $\mathcal P_S$. The problem of showing that there are {\it any} of these primes, let alone infinitely many, seems a difficult one, though computationally they are easy to find for a particular $P$ and $Q$.

We now consider the finiteness (or otherwise) of $T$ and $\mathcal P_T$.

\begin{theorem} [{{Somer \cite[Theorems 8,9]{MR1393479}}}]\label{T-finite}
The set $T$ is finite in the following two cases:
\begin{itemize}
\item $P=\pm 1$, $Q\not\equiv -1\pmod{6}$, in which case $T=\{1\}$;
\item $P=\varepsilon_1 2^k$, $Q=2^{2k-1}+\varepsilon_2$, where $k$ is a positive integer, and $\varepsilon_1$,
$\varepsilon_2\in\{-1,1\}$, in which case $T=\{1,2\}$.
\end{itemize}
Otherwise, $T$ is infinite. For $T$ infinite, $\mathcal P_T$ is
finite precisely when $P,Q$ are not both $0$ and either 
\begin{itemize}
\item $P^2=Q$, in which case $\mathcal P_T$ is the set of prime divisors of $2P$

or 
\item $P^2=4Q$ or $Q=0$, in which case $\mathcal P_T$ is the set of prime divisors of $P$.  
\end{itemize}
Otherwise, for $T$ infinite, $\mathcal P_T$ is also infinite. 
\end{theorem}

\begin{proof} If $T$ contains an integer $n$
having an odd prime factor $p$ then, by Theorem \ref{T-Basic}(a),
$p^kn\in T$ for all $k\ge 0$. In particular, if $P=\pm 1$ and
$Q\equiv -1\pmod{6}$, then $6\in T$, so that $T$ is infinite. On the
other hand, if $P=\pm 1$ and $Q\not\equiv -1\pmod{6}$, then $1$ is
the only basic element of $T$, and $v_1=P$ has no prime factors so
that, by Theorem \ref{T-Basic}(a), $\mathcal P_1$ is empty, and
hence $T=\{1\}$.

Again starting with $1\in T$, we see that $T$ is infinite if $P$ has any odd prime factors. Also, $T$ is infinite if $P$ is $\pm$ a positive power of $2$ and $2$ is special, as then $2^k\in T$ for all $k\ge 0$, by Theorem \ref{T-Basic}(a).

It therefore remains only to consider the case of $P=\pm 2^k$, $k\ge 1$ and $Q$ odd, so that $2$ is not special.
 Then $2\in T$ and $4\notin T$, by Theorem \ref{T-Basic}(a). If $v_2$ has an odd prime factor $p$,
 then $2p^k\in T$ for all $k\ge 0$, so that $T$ is again infinite. Finally, if $v_2$ is $\pm$ a power of $2$,
 then $T=\{1,2\}$. This happens only when $v_2=2^{2k}-2Q=\pm 2$, so that $Q=2^{2k-1}\mp 1$, as claimed.
 
 Now take $T$ infinite, with $P,Q$ not both $0$. If the sequence $(v_n)$ is degenerate, then, using Somer \cite[Theorem 9]{MR1393479}, we get either $P^2=Q$, $P^2=4Q$ or $Q=0$, and $\mathcal P_T$ being the set of prime divisors of $P$, as required. On the other hand, if $(v_n)$ is not degenerate then by Somer \cite[Theorem
1]{MR1393479} for sufficiently large $n$ every $v_n$ has a primitive prime divisor. Hence we can find an infinite sequence of numbers $n$ in  $T$ such that $np$ is again in $T$, where $p$ is a primitive prime divisor of $v_n$.
(Here we are using Theorem \ref{T-Basic}(a).) Thus $\mathcal P_T$ then contains infinitely many primes $p$.
 \end{proof}

\section{Divisibility properties of $S$ and of $T$.} \label{S-divisibility}
From Theorem \ref{S-Basic} we can consider $S$ as spanned by
a forest of one or two trees, with each node corresponding to an element of $S$, and the root nodes being $\{1\}$,
 $\{1,6\}$ or $\{1,12\}$. Each edge can be labelled $p$; it rises from a node  $n\in S$ to a node $np\in S$, where $p$ is some prime divisor  of $u_n\Delta$. Thus every node above $n$ in the tree is divisible by $n$. Then call a {\it cutset} of the forest a set $C$ of nodes with the property that every path from a root to infinity must contain some vertex of the cutset. Then we clearly have the following.
\begin{prop} For a cutset $C$ of $S$, every element of $S$ either lies below $C$, or it is divisible by some node
of $C$.
\end{prop}
Judicious choice of a cutset places severe divisibility restrictions on elements of $S$, and so, using this, one can search for elements of $S$ up to an given bound very efficiently. 

The same argument applies equally to $T$, using Theorem
\ref{T-Basic}, with $p$ being either an odd prime divisior of $v_n$
or, under the conditions described in that theorem, the prime $2$. For instance, applying this idea to the sequence $T$ of example 2 below, every element of that sequence not a power of 3 is divisible either by 171 or 243 or 13203 or 2354697 or 10970073 or 22032887841. See \cite{BS} for details.

\section{Examples}\label{S-Ex}
\begin{enumerate}
\item[1.] $P=1,Q=-1$ (the classical Fibonacci and Lucas numbers.) Here $\Delta=5$,
$$
S=1, 5, 12, 24, 25, 36, 48, 60, 72, 96, 108, 120, 125, 144, 168, 180,\dots ,
$$
with $1$ and $12$ basic (A023172 on Neil Sloane's Integer Sequence website), while $\mathcal P_S$ is the whole of $\mathcal P$ (see Theorem \ref{infinite}),
$$
T=1, 6, 18, 54, 162, 486, 1458, 1926, 4374, 5778, 13122, 17334, \dots,
$$
with $1$ and $6$ basic (A016089), and
$$
\mathcal P_{\text{2nd}} =2, 3, 7, 11, 19, 23, 29, 31, 41, 43, 47, 59, 67, 71, 79, 83, 101, 103, 107, 127,\dots,
$$
(A140409) of which $\mathcal P_T$ is a subsequence:
$$
\mathcal P_T=2, 3, 107, 1283, 8747, 21401, 34667, 46187,\dots,
$$
(A016089, see (see Theorem \ref{T-finite}).
\item[2.] $P=3, Q=2$, where $u_n=2^n-1$, $v_n=2^n+1$. Here $S=\{1\}$ as $\Delta=1$, and
$$
T=1, 3, 9, 27, 81, 171, 243, 513, 729, 1539, 2187, 3249,\dots,
$$
with $1$ basic (A006521).   Also
$$
\mathcal P_{\text{2nd}} =3, 5, 11, 13, 17, 19, 29, 37, 41, 43, 53, 59, 61, 67, 83, 97, 101, 107, 109,\dots,
$$
(A014662 -- see also A091317), of which
$$
\mathcal P_T=3, 19, 163, 571, 1459, 8803, 9137, 17497, 41113, \dots
$$
(A057719) is a subsequence. Note too that, by Proposition \ref{P-AJ} and the fact that all $n\in T$ are odd, we have $T=S(-1,-2)$. Also $S=T(-1,-2)=\{1\}$.

\item[3.] $P=3$, $Q=5$, $\Delta=-11$,
$$
S=1, 6, 11, 12, 18, 24, 36, 48, 54, 66, 72, 96, 108, 121, 132, 144, 162, 168, 192, 198,\dots
$$
with $1$ and $6$ basic, with $\mathcal P_{\operatorname{1st}}$ consisting of all primes except the irregular prime $5$, and
$$
\mathcal P_S= 2, 3, 7, 11, 13, 17, 23, 37, 41, 43, 67, 71, 73, 83, 89, 97, 101, 103, 107, 113, \dots .
$$
Also 
$$
T=1, 3, 9, 27, 81, 153, 243, 459, 729, 1377, 2187, 2601, 4131, 4401, 6561, 7803, \dots
$$
with only $1$ basic, 
$$
\mathcal P_{\operatorname{2nd}} = 2, 3, 7, 13, 17, 19, 23, 37, 43, 47, 53, 67, 73, 79, 83, 97, 103, 107, 113, \dots
$$
and
$$
\mathcal P_T = 2, 3, 17, 103, 163, 373, 487, 1733,\dots .
$$
\end{enumerate}

\section{Final remarks.}\label{S-final}

\begin{enumerate}
\item[1.] It would be interesting to see whether the analysis of the paper could be extended to other second-order recurrence sequences, or indeed to any recurrences of higher order.
\item[2.] In \cite{BS}, what we called `primitive' solutions of $n\mid 2^n+1$ should in fact have been called {\it fundamental} solutions, following Jarden \cite[p. 70]{MR0197383} and Somer \cite[p. 522]{MR1271392}, \cite[p. 482]{MR1393479}. However, this definition has been superseded by the notion of a basic element (of $S$ or of $T$) as in this paper.
\item[3.] In example 1 of Section \ref{S-Ex} above we have $24$ and $25\in S=S(1,-1)$. Are these the only consecutive integers in $S(1,-1)$?
\end{enumerate}
\bibliographystyle{plain}
\bibliography{Smyth}

\end{document}